\documentclass[12pt,fleqn]{article}
\usepackage{amsmath,amsthm,amssymb,amscd}
\usepackage{graphicx}
\usepackage{graphics}
\usepackage{verbatim}

\usepackage{mathrsfs}
\usepackage{color}
\renewcommand{\baselinestretch}{1.4}

\makeatletter

\newcommand{\Rmnum}[1]{\expandafter\@slowromancap\romannumeral #1@}

\newtheorem{theorem}{Theorem}[section]
\newtheorem{proposition}[theorem]{Proposition}

\newtheorem{problem}[theorem]{Problem}
\newtheorem{lemma}[theorem]{Lemma}

\newtheorem{claim}{Claim}[theorem]

\makeatother

\begin{document}
\title{Face-degree bounds for planar critical graphs}

\author{Ligang Jin\footnotemark[1] \footnotemark[2] , Yingli Kang\footnotemark[1] \footnotemark[3] , Eckhard Steffen\footnotemark[1]}

\footnotetext[1]{Institute of Mathematics and Paderborn Center for Advanced Studies,
	Paderborn University, Paderborn, Germany}
\footnotetext[2]{supported by Deutsche Forschungsgemeinschaft (DFG) grant STE 792/2-1}
\footnotetext[3]{Fellow of the International Graduate School "Dynamic Intelligent Systems"}

\footnotetext{Email: \{ligang, yingli\}@mail.upb.de, es@upb.de}					

\date{}

\maketitle

\abstract{The only remaining case of a well known conjecture of Vizing states that there is no planar graph with maximum degree 6 and 
	edge chromatic number 7. 
	We introduce parameters for planar graphs,  based on the degrees of the faces, and study the question whether
	there are upper bounds for these parameters for planar edge-chromatic critical graphs. Our results provide
	upper bounds on these parameters for smallest counterexamples to Vizing's conjecture, thus providing a partial characterization 
	of such graphs, if they exist. 
	
	For $k \leq 5$ the results give insights into the structure of planar edge-chromatic critical graphs.}


\section{Introduction}

We consider finite simple graphs G with vertex set $V(G)$ and edge set $E(G)$. The \textit{vertex-degree} of $v \in V(G)$ is denoted by $d_G(v)$, and $\Delta(G)$ denotes the \textit{maximum vertex-degree} of $G$.
If it is clear from the context, then $\Delta$ is frequently used.
The \textit{edge-chromatic-number} of $G$ is denoted by $\chi'(G)$. Vizing \cite{Vizing_1964} proved that $\chi'(G) \in \{\Delta(G), \Delta(G)+1\}$. If $\chi'(G) = \Delta(G)$, then $G$ is a \textit{class 1} graph,
otherwise it is a \textit{class 2} graph. A class 2 graph $G$ is \textit{critical}, if $\chi'(H) < \chi'(G)$ for every proper subgraph $H$ of $G$. Critical graphs with maximum vertex-degree $\Delta$ are also
called \textit{$\Delta$-critical}. It is easy to see that critical graphs are 2-connected.
A graph $G$ is \textit{overfull} if $|V(G)|$ is odd and $|E(G)| \geq  \Delta(G) \lfloor \frac{1}{2} |V(G)| \rfloor + 1$.
Clearly, every overfull graph is class 2.
A graph is \textit{planar} if it can be embedded into the Euclidean plane. A \textit{plane graph} $(G,\Sigma)$ is a planar graph $G$ together with an embedding $\Sigma$ of $G$ into the Euclidean plane.
That is, $(G,\Sigma)$ is a particular drawing of $G$ in the Euclidean plane.

In 1964, Vizing \cite{Vizing_1964} showed for each $k \in \{2,3,4,5\}$ that there is a planar class 2 graph $G$ with $\Delta(G) = k$. He proved that every planar graph with
maximum vertex-degree at least 8 is a class 1 graph, and conjectured that every planar graph $H$ with $\Delta(H) \in \{6,7\}$ is a class 1 graph. Vizing's conjecture has been
proved for planar graph with maximum vertex-degree 7 by Gr\"unewald \cite{Gruenewald_2000},  Sanders and Zhao \cite{Sanders_Zhao_2001}, and Zhang \cite{Zhang_2000} independently.

Zhou \cite{Zhou_2003} proved for each $k \in \{3,4,5\}$ that if $G$ is a planar graph with $\Delta(G) = 6$ and $G$ does not contain a circuit of length $k$, then $G$ is a class 1 graph.
Vizing's conjecture is confirmed for some other classes of planar graphs which do not contain some specific (chordal) circuits \cite{Bu_Wang_2006, Wang_Chen_2007, Wang_Xu_2013}.

\begin{figure} [ht]
	\centering
	\includegraphics[width=4.0in]{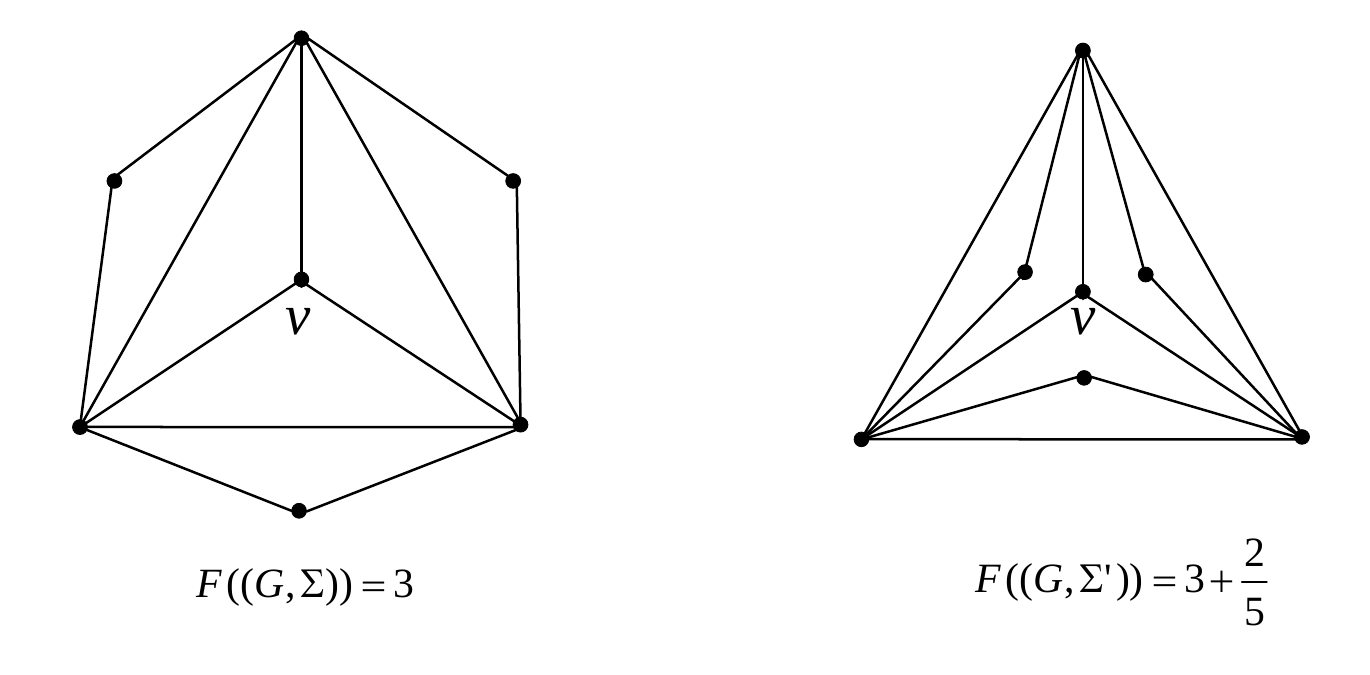}
	\caption{Graph $G$ has two embeddings $\Sigma$, $\Sigma'$ such that $F((G,\Sigma)) \not = F((G,\Sigma'))$.} \label{F_overline}
\end{figure}

Let $G$ be a 2-connected planar graph, $\Sigma$ be an embedding of $G$ in the Euclidean plane and $F(G)$ be the set of faces of $(G,\Sigma)$.
The \textit{degree} $d_{(G,\Sigma)} (f) $ of a face $f$ is the length of its facial circuit.
If there is no harm of confusion we also write $d_G(f)$ instead of $d_{(G,\Sigma)}(f)$.
Let
$\overline{F}(G) = \frac{1}{|F(G)|} \sum_{f \in F(G)} d_{G}(f)$ be the {\em average face-degree} of $G$.
Euler's formula $|V(G)| - |E(G)| + |F(G)| = 2$ implies that $\overline{F}(G) = \frac{2|E(G)|}{|E(G)|-|V(G)|+2}$.

Let $v \in V(G)$. If $d_G(v) = k$, then $v$ is incident to $k$ pairwise different
faces $f_1, \dots, f_k$.  Let $F_{(G,\Sigma)}(v)=\frac{1}{k}(d_{(G,\Sigma)}(f_1)+ \dots + d_{(G,\Sigma)}(f_k))$ and $F((G,\Sigma)) = \min\{F_{(G,\Sigma)}(v)\colon\ v\in V(G)\}$.
Clearly, $F((G,\Sigma)) \geq 3$ since every face has length at least 3. As Figure \ref{F_overline} shows, $F((G,\Sigma))$ depends on the embedding $\Sigma$.
The {\em local average face-degree} of a 2-connected planar graph $G$ is 
$$F^*(G) = \max \{F((G,\Sigma))\colon\ (G,\Sigma) \mbox{ is a plane graph}\}.$$
This parameter is independent from the embeddings of $G$, and $F^*(G)\geq 3$ for all planar graphs.
Let $k$ be a positive integer. Let
$\overline{b}_k = \sup \{\overline{F}(G)\colon\ G \mbox{ is a $k$-critical planar graph}\}$ and $b^*_k = \sup \{F^*(G)\colon\ G \mbox{ is a $k$-critical planar graph}\}$.
We call $\overline{b}_k$ the {\em the average face-degree bound}, and $b^*_k$ the
{\em local average face-degree bound} for $k$-critical planar graphs.
If $k = 1$ or $k \geq 7$, then every planar graph with maximum vertex-degree $k$ is a class 1 graph and therefore,
$\{\overline{F}(G)\colon\ G \mbox{ is a $k$-critical planar graph}\}=\{F^*(G)\colon\ G$ is a $k$-critical planar graph$\} = \emptyset$.
Hence, $\overline{b}_k$ and $b^*_k$ do not exist in these cases. Therefore, we focus
on $k \in \{2,3,4,5,6\}$ in this paper. The main results are the following two theorems.

\begin{theorem} \label{Main_global} Let $k \geq 2$ be an integer.
	\begin{itemize}
		\item If $k = 2$, then $\overline{b}_k = \infty$.
		\item If $k = 3$, then $6 \leq \overline{b}_k \leq 8$.
		\item If $k=4$, then $4 \leq \overline{b}_k \leq 4 + \frac{4}{5}$
		\item If $k=5$, then $3 + \frac{1}{3} \leq \overline{b}_k \leq 3 + \frac{3}{4}$.
		\item If $k= 6$ and $\overline{b}_k$ exists, then $\overline{b}_k \leq 3 + \frac{1}{3}$.
	\end{itemize}
\end{theorem}

\begin{theorem} \label{MAIN} Let $k \geq 2$ be an integer.
	\begin{itemize}
		\item If $k \in \{2,3,4\}$, then $b^*_k = \infty$.
		\item If $k = 5$, then $3 + \frac{1}{5} \leq b^*_k \leq 7+\frac{1}{2}$.
		\item If $k = 6$ and $b^*_k$ exists, then $b^*_k \leq 3 + \frac{2}{5}$.
	\end{itemize}
\end{theorem}

Vizing \cite{Vizing_1965} proved that a class 2 graph contains $k$-critical subgraph for each $k \in \{2, \dots,\Delta\}$. 
Hence a smallest counterexample to Vizing's conjecture is critical and thus, our results for $k=6$
partially characterize smallest counterexamples to this conjecture. For $k \leq 5$, they provide insight into the structure of 
planar critical graphs.
Seymour's exact conjecture \cite{Stiebitz_2012} says that every critical planar graph is overfull.
If this conjecture is true for $k\in \{3,4,5\}$, then $\overline{b}_k$ is equal to the lower bound given in Theorem \ref{Main_global}.

It is not clear whether $\overline{b}_k$ and $b_k^*$ or $\overline{F}(G)$ and $F^*{(G)}$ are related to each other, respectively. 
Furthermore, the precise values of $\overline{b}_k$ and $b^*_k$ are also unknown.

The next section states some properties of critical and of planar graphs. These results are
used for the proofs of Theorems \ref{Main_global} and \ref{MAIN} which are given in Section \ref{proofs}.

\section{Preliminaries}

Let $G$ be a 2-connected graph. A vertex $v$ is called a \textit{$k$-vertex}, or a \textit{$k^+$-vertex}, or a \textit{$k^-$-vertex} if $d_G(v)=k$, or $d_G(v)\geq k$, or $d_G(v)\leq k$, respectively.
Let $N(v)$ be the set of vertices which are adjacent to $v$, and $N(S)=\bigcup_{v\in S}N(v)$ for a set $S\subseteq V(G)$.
We write $N(v)$ and $N(u,v)$ short for $N(\{v\})$ and $N(\{u,v\})$, respectively.

Let $(G,\Sigma)$ be a plane graph. A face $f$ is called \textit{$k$-face}, or a \textit{$k^+$-face}, or a \textit{$k^-$-face}, if $d_{(G,\Sigma)}(f)=k$, or $d_{(G,\Sigma)}(f)\geq k$, or
$d_{(G,\Sigma)}(f)\leq k$, respectively. We will use the following well-known results on critical graphs.

\begin{lemma} \label{trivial_statement}
	Let $G$ be a critical graph and $e \in E(G)$. If $e = xy$, then $d_G(x) \geq 2$, and $d_G(x)+d_G(y) \geq \Delta(G)+2$.
\end{lemma}

\begin{lemma} [Vizing's Adjacency Lemma \cite{Vizing_1964}] \label{VAL}
	Let $G$ be a critical graph. If $e=xy \in E(G)$, then at least $\Delta(G)-d_G(y)+1$ vertices in $N(x)\setminus \{y\}$ have degree  $\Delta(G)$.
\end{lemma}

\begin{lemma} [\cite{Zhang_2000}] \label{Zhang_lemma}
	Let $G$ be a critical graph and $xy \in E(G)$. If $d(x)+d(y) = \Delta(G)+2$, then
	\begin{enumerate}
		\item every vertex of $N(x,y) \setminus \{x,y\}$ is a $\Delta(G)$-vertex,
		\item every vertex in $N(N(x,y)) \backslash \{x,y\}$ has degree at least $\Delta(G)-1$,
		\item if $d(x)< \Delta(G)$ and $d(y)< \Delta(G)$, then every vertex in $N(N(x,y)) \backslash \{x,y\}$ has degree $\Delta(G)$.
	\end{enumerate}
\end{lemma}

\begin{lemma}  [\cite{Sanders_Zhao_2001}] \label{SZ}
	No critical graph has pairwise distinct vertices $x, y,z$, such that $x$ is adjacent to $y$ and $z$, $d(z)<2\Delta(G)-d(x)-d(y)+2$, and xz is in at least $d(x)+d(y)-\Delta(G)-2$ triangles not containing $y$.
\end{lemma}

We will use the following results on lower bounds for the number of edges in critical graphs.

\begin{theorem} [\cite{Jakobsen_1974}] \label{Jakobsen}
	If $G$ is a $3$-critical graph, then $|E(G)| \geq \frac{4}{3}|V(G)|$.
\end{theorem}

\begin{theorem} [\cite{Woodall_2008}] \label{Woodall}
	Let $G$ be a $k$-critical graph. If $k = 4$, then $|E(G| \geq \frac{12}{7}|V(G)|$, and if $k = 5$, then $|E(G)| \geq \frac{15}{7}|V(G)|$.
\end{theorem}

\begin{theorem} [\cite{Luo_Miao_Zhao_2009}] \label{Luo_etal}
	If $G$ is a $6$-critical graph, then $|E(G)| \geq \frac{1}{2}(5|V(G)|+3)$.
\end{theorem}

\begin{lemma} \label{graphs} Let $t$ be a positive integer and $\epsilon > 0$.
	\begin{enumerate}
		\item For $k \in \{2,3,4\}$ there is a $k$-critical planar graph $G$ and $F^*(G) > t$.
		\item There is a $2$-critical planar graph $G$ with $ \overline{F}(G) > t$.
		\item There is a 3-critical planar graph $G$ such that  $ 6 - \epsilon < \overline{F}(G) < 6$.
		\item There is a 4-critical planar graph $G$  such that $ 4 - \epsilon < \overline{F}(G) < 4$.
		\item There is a 5-critical planar graph $G$, such that $3 + \frac{1}{3} - \epsilon < \overline{F}(G) < 3 + \frac{1}{3}$ and $F^*(G) \geq 3 + \frac{1}{5}$.
	\end{enumerate}
\end{lemma}

\begin{proof} The odd circuits are the only 2-critical graphs. Hence, the second statement and the first statement for $k=2$ are proved.
	Let $X$ and $Y$ be two circuits of length $n \geq 3$, with $V(X) = \{x_i\colon\ 0\leq i\leq n-1\}$, $V(Y) = \{y_i\colon\ 0\leq i\leq n-1\}$ and edges
	$x_ix_{i+1}$ and $y_iy_{i+1}$, where the indices are added modulo $n$. Consider an embedding, where $Y$ is inside $X$. Add edges
	$x_iy_i$ to obtain a planar cubic graph $G$ with $F^*(G) = \frac{1}{3}(n+8)$. Add edges $x_iy_{i+1}$ in $G$ to obtain a
	4-regular planar graph $H$ with $F^*(H) = \frac{1}{4}(n + 9)$.  Subdividing one edge in $G$ and one in $H$ yields a critical planar graph $G_n$ with $\Delta(G_n)=3$, and a critical planar graph $H_n$ with $\Delta(H_n) = 4$. If $n\geq 4t$, then $F^*(G_n) > t$ and $F^*(H_n) > t$. The proof that $G_n$ and $H_n$ are critical
	will be given in the last paragraph.  
	
	Since $|F(G_n)| = n+2$, and $\sum_{f \in F(G_n)}d_{G_n}(f) = 6n+2$, it follows that $\overline{F}(G_n) = 6 - \frac{10}{n+2}$. Analogously, we have $|F(H_n)| = 2n+2$ and $\sum_{f \in F(H_n)}d_{H_n}(f) = 8n+2$ and therefore, $\overline{F}(H_n) = 4 - \frac{3}{n+1}$.
	Now, the statements for $3$-critical and $4$-critical graphs follow.
	Examples of these graphs are given in Figure \ref{no_b_k_234}.
	
	\begin{figure} [ht]
		\centering
		\includegraphics[width=4.5in]{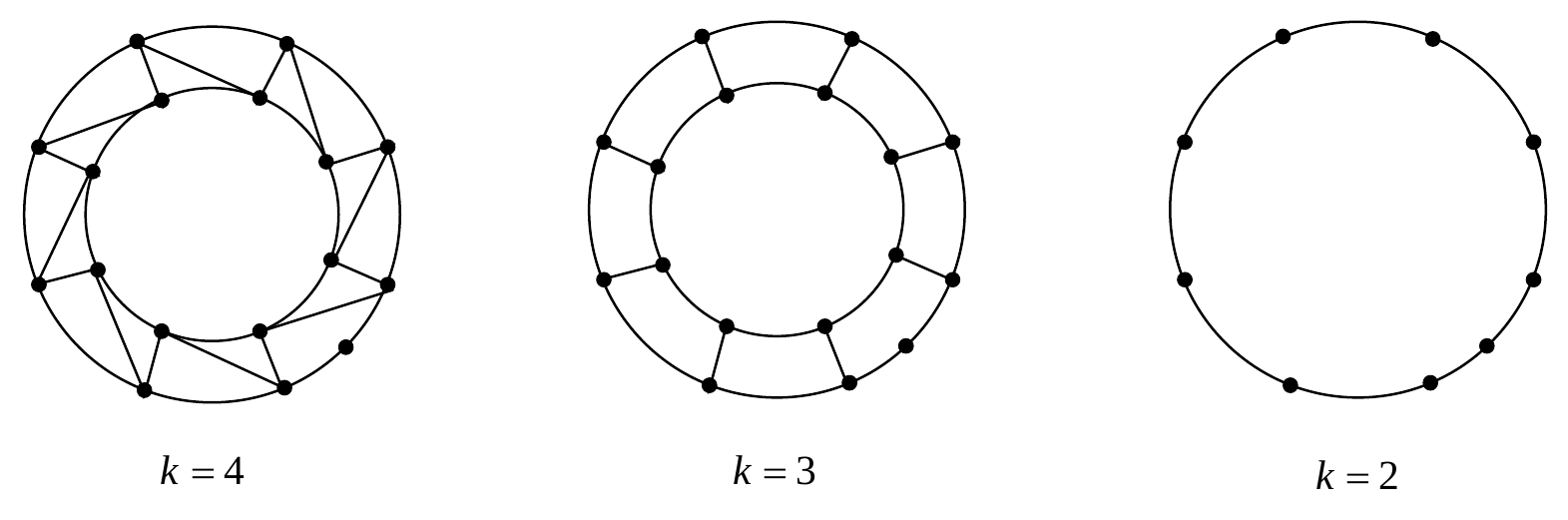}
		\caption{Examples for $k \in \{2,3,4\}$} \label{no_b_k_234}
	\end{figure}
	
	Let $m\geq 4$ be an integer. Let $C_i=[c_{i,1}c_{i,2}\cdots c_{i,4}]$ be a circuit of length 4 for $i\in \{1,m\}$, and $C_i=[c_{i,1}c_{i,2}\cdots c_{i,8}]$ be a circuit of length 8 for $i \in \{2, \dots, m-1\}$.
	Consider an embedding, where $C_i$ is inside $C_{i+1}$ for $ i \in \{  1, \dots ,m-1\}$.
	Add edges $c_{1,j}c_{2,2j-1}$, $c_{1,j}c_{2,2j}$, $c_{1,j}c_{2,2j+1}$ for $j\in \{1, \dots, 4\}$, edges $c_{i,j}c_{i+1,j}$
	for $i \in \{2, \dots, m-2\}$ and $j \in \{1, \dots, 8\}$, edges $c_{i,j}c_{i+1,j+1}$ for $ i \in \{2, \dots, m-2\}$ and $j\in \{2,4,6,8\}$, and
	edges $c_{m-1,2j-2}c_{m,j}$, $c_{m-1,2j-1}c_{m,j}$ and $c_{m-1,2j}c_{m,j}$ for $j \in \{1, \dots, 4\}$ to obtain a 5-regular planar graph $T$
	(the indices are added modulo 8). Subdividing the edge $c_{m,1}c_{m,2}$ in $T$ yields a critical planar graph $T_m$ with $\Delta(T_m)=5$
	(Figure \ref{lower_bound_b5} illustrates $T_6$).
	
	\begin{figure} [ht]
		\centering
		\includegraphics[width=2.2in]{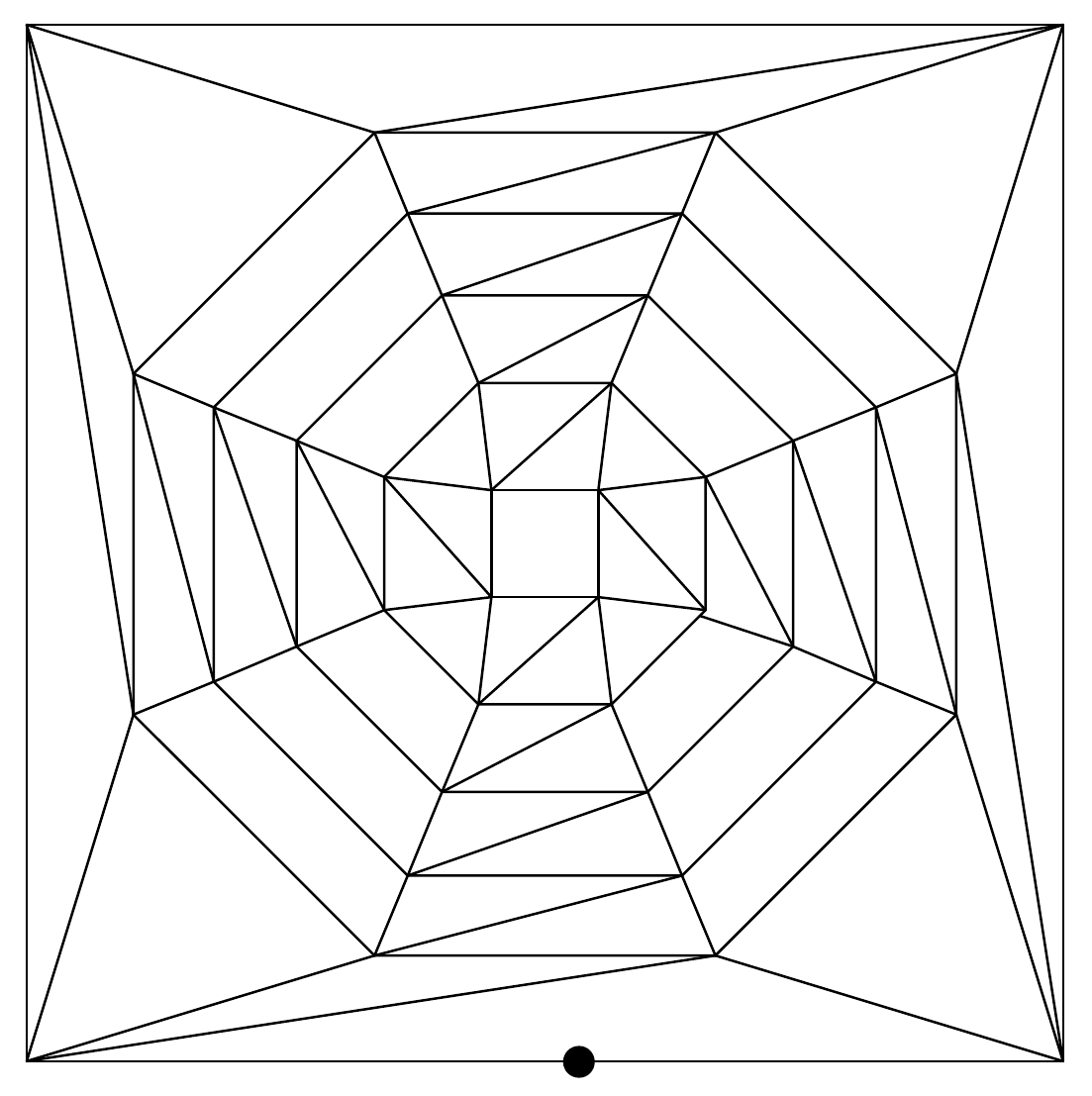}
		\caption{$(T_6,\Sigma_6)$} \label{lower_bound_b5}
	\end{figure}
	
	Since $|F(T_m)|=12m-10$ and $\sum_{f\in F(T_m)}d_{T_m}(f)=40m-38$, it follows that $\overline{F}(T_m)=\frac{10}{3}-\frac{7}{18m-15}$. 
	Furthermore, for the embedding $\Sigma_m$ of $T_m$ as indicated in Figure \ref{lower_bound_b5} (for $m=6$) 
	we calculate that $F((T_m,\Sigma_m)) = 3 + \frac{1}{5}$
	and therefore, $F^*(T_m) \geq 3 + \frac{1}{5}$.
	
	It remains to prove that $G_n$, $H_n$ and $T_m$ are critical. 
	For $G_n$ and $H_n$ we proceed by induction on $n$.
	It is easy to verify the truth for $3\leq n\leq 6$. We proceed to induction step.
	We argue first on $G_n$. Let $u$ be the vertex of degree 2. Since $n\geq 7$, for any edge $e$ of $G_n$, there exists some $k$ such that no vertex of the circuit $C$  is incident with $e$ or adjacent to $u$, where $C=[x_{k+1}y_{k+1}y_{k+2}x_{k+2}]$.
	Reduce $G_n$ to $G_{n-2}$ by removing the edges $x_{k+1}y_{k+1}$ and $x_{k+2}y_{k+2}$ and suppressing their ends. Let $G'$ be the resulting graph and $e'$ be the resulting edge from $e$. By the induction hypothesis, $G'$ is critical. Hence, $G'-e'$ has a 3-edge-coloring, say $\phi$.  
	Assign $\phi (x_kx_{k+3})$ to $x_kx_{k+1}$ and $x_{k+2}x_{k+3}$, and $\phi (y_ky_{k+3})$ to $y_ky_{k+1}$ and $y_{k+2}y_{k+3}$, and consequently, the edges of $C$ can be properly colored. Now a 3-edge-coloring of $G_n-e$ is completed and so, $G_n-e$ is class 1. Moreover, since $G_n$ is overfull, this graph is class 2. Therefore, $G_n$ is critical.
	The argument on $H_n$ is analogous.
	
	For any $T_m$, recall that $T$ is the graph obtained from $T_m$ by suppressing the bivalent vertex.
	Consider $T$. Since each circuit $C_i$ has even length, their edges can be decomposed into two perfect matchings $M_1$ and $M_2$, so that $M_1$ contains $c_{i,1}c_{i,2}$ for $i\in \{1,m\}$ and  $c_{i,2}c_{i,3}$ for $2\leq i\leq m-1$.
	Let $M_3=\{c_{1,j}c_{2,2j+1}\colon\ 1\leq j\leq 4\}\cup \{c_{i,2j}c_{i+1,2j+1}\colon\ 2\leq i\leq m-2,1\leq j\leq 4\}\cup \{c_{m-1,2j-2}c_{m,j}\colon\ 1\leq j\leq 4\}$.
	Clearly, $M_3$ is a perfect matching disjoint with $M_1$ and $M_2$. We can see that $E(G)\setminus (M_1\cup M_2\cup M_3)$ induces even circuits and hence, their edges can be decomposed into two perfect matchings $M_4$ and $M_5$, so that $M_4$ contains $c_{1,j}c_{2,2j}$ for $1\leq j\leq 4$. 
	Clearly, $M_1,\ldots,M_5$ constitute a decomposition of $E(T)$.
	
	Let $e_i=c_{m,i}c_{m,i+1}$ for $1\leq i\leq 4$. Let $M_2'=M_2\cup \{e_1,e_3\}\setminus \{e_2,e_4\}$.
	Define $A_1 = M_1\cup M_3, A_2 = M_2'\cup M_4, A_3 = M_2'\cup M_5.$
	
	Let $h_m$ be an edge of $T_m$.
	Since $T_m$ is overfull, to prove that $T_m$ is critical, it suffices to show that $T_m-h_m$ is a 5-edge-colorable. 
	
	Let $h$ be the edge of $T$ that corresponds to $h_m$.
	We can see that $A_1\cup A_2\cup A_3= E(T)\setminus \{e_2,e_4\}$ and $e_1\in A_1\cap A_2\cap A_3$.
	Hence, if $h\notin \{e_2,e_4\}$ then there exists $A\in \{A_1,A_2,A_3\}$ such that $e_1,h\in A$.
	Note that $e_1$ is the edge subdivided to get $T_m$ from $T$, and that $A$ induces a circuit of $T$. It follows that this circuit corresponds to a path $P$ of $T_m-h_m$.
	Moreover, note that the edges of $T-A$ can be decomposed into 3 perfect matchings, and thus the same to the edges of  $T_m-h_m-E(P)$. Therefore, $T_m-h_m$ is 5-edge-colorable.
	
	If $h\in \{e_2,e_4\}$ then $C_m$ corresponds to a path of $T_m-h_m$.
	Note that $E(C_m)\subseteq M_1\cup M_2$ and that $M_1,\ldots,M_5$ constitute a decomposition of $E(T)$. Similarly, we can argue that $T_m-h_m$ is 5-edge-colorable in this case.
\end{proof}

The following lemma is implied by Euler's formula directly.

\begin{lemma} \label{size_planar}
	If $G$ is a planar graph, then $|E(G)| = \frac{\overline{F}(G)}{\overline{F}(G)-2}(|V(G)|-2)$.
\end{lemma}

\section{Proofs} \label{proofs}

\subsection{Theorem \ref{Main_global}}

The statement for $k=2$ and the lower bounds for $\overline{b}_k$ if $k \in \{3,4,5\}$ follow from Lemma \ref{graphs}.
The other statements of Theorem \ref{Main_global} are implied by the following proposition.

\begin{proposition} Let $G$ be a $k$-critical planar graph. \label{pro_globe}
	\begin{enumerate}
		\item If $k = 3$, then $\overline{F}(G) < 8 $.
		\item If $k = 4$, then $\overline{F}(G) < 4 + \frac{4}{5}$.
		\item If $k = 5$, then $\overline{F}(G) < 3 + \frac{3}{4}$.
		\item If $k = 6$, then $\overline{F}(G) < 3 + \frac{1}{3}$.
	\end{enumerate}
\end{proposition}

\begin{proof}
	Let $k = 3$ and suppose to the contrary that $\overline{F}(G) \geq 8$. With Lemma \ref{size_planar} and Theorem \ref{Jakobsen} we deduce
	$ \frac{4}{3}|V(G)| \leq |E(G)| \leq \frac{4}{3}(|V(G)|-2)$, a contradiction.
	
	The other statements follow analogously from Lemma \ref{size_planar} and Theorem \ref{Woodall} ($k \in \{4,5\}$) and Theorem \ref{Luo_etal} ($k=6$).
\end{proof}

\subsection{Theorem \ref{MAIN}}

The statement for $k \in \{2,3,4\}$ and the lower bound for $b^*_5$ follow from Lemma \ref{graphs}. It remains to prove the
upper bounds for $b^*_5$ and $b^*_6$. The result for $b^*_5$ is implied by the following theorem.

\begin{theorem} \label{thm_delta5_F*}
	If $G$ is a planar $5$-critical graph, then $F^*(G) \leq 7 + \frac{1}{2}$.
\end{theorem}

\begin{proof}
	Suppose to the contrary that $F^*(G) =r > 7 + \frac{1}{2}$.
	Let $\Sigma$ be an embedding of $G$ into the Euclidean plane such that $F^*(G) = F((G,\Sigma))$.  Let $V=V(G)$, $E=E(G)$, and $F$ be the set of faces of $(G,\Sigma)$.
	We proceed by a discharging argument in $G$ and eventually deduce a contradiction.
	Define the initial charge $ch$ in $G$ as $ch(x)=d_G(x)-4$ for $x\in V\cup F$. Euler's formula $|V|-|E|+|F|=2$ can be rewritten as:
	$$\sum\limits_{x\in V\cup F}ch(x)=\sum\limits_{x\in V\cup F}(d_G(x)-4)=-8.$$
	
	We define suitable discharging rules to change the initial charge function $ch$ to the final charge function $ch^*$ on
	$V\cup F$ such that $\sum\limits_{x\in V\cup F}ch^*(x) \geq 0$ for all $x\in V\cup F$. Thus,
	\begin{center}
		$-8=\sum\limits_{x\in V\cup F}ch(x)=\sum\limits_{x\in V\cup F}ch^*(x)\geq0$,
	\end{center}
	which is the desired contradiction.
	
	Note that if a face $f$ sends charge $-\frac{1}{3}$ to a vertex $y$, then this can also be considered as $f$ receives charge $\frac{1}{3}$ from $y$. The discharging rules are defined as follows.
	
	\noindent
	\textbf{R1:} Every $3^+$-face $f$ sends $\frac{d_G(f)-4}{d_G(f)}$ to each incident vertex.
	
	\noindent
	\textbf{R2:} Let $y$ be a 5-vertex of $G$.
	
	{\bf R2.1:} If $z$ is a 2-neighbor of $y$, then $y$ sends $\frac{2}{3}+\frac{2}{\lceil 2r \rceil-3}$ to $z$.
	
	{\bf R2.2:} If $z$ is a 3-neighbor of $y$, then $y$ sends charge to $z$ as follows:
	
	~\quad{\bf R2.2.1:} if $z$ has a $4$-neighbor, then $y$ sends $\frac{1}{3}+\frac{2}{\lceil 3r \rceil-6}$ to $z$;
	
	~\quad{\bf R2.2.2:} if $z$ has no $4$-neighbor, then $y$ sends $\frac{2}{9}+\frac{4}{3(\lceil 3r \rceil-6)}$ to $z$.
	
	{\bf R2.3:} If $z$ is a 4-neighbor of $y$ and $z$ is adjacent to $n$ 5-vertices $(2\leq n\leq 4)$, then $y$ sends $\frac{4}{n(\lceil 4r \rceil-9)}$ to $z$.
	
	{\bf R2.4:} If $y$ is adjacent to five $4^+$-vertices, then $y$ sends $\frac{1}{3}(\frac{4}{\lceil 5r \rceil-12}+\frac{2}{\lceil 2r \rceil-3})$ to each 5-neighbor which is adjacent to a 2-vertex.
	
	\begin{claim} \label{C1}
		If $u$ is a $k$-vertex, then $u$ receives at least $\frac{4-k}{3}-\frac{4}{\lceil rk \rceil-3k+3}$ in total from its incident faces by R1.
		In particular, if $u$ is incident with at most two triangles, then $u$ receives at least $\frac{1}{3}-\frac{4}{\lceil rk \rceil -4k+6}$ in total from its incident faces.
	\end{claim}
	
	\begin{proof} Note that if $a$ and $b$ are integers and $2 \leq a \leq b$, then $\frac{1}{a-1}+\frac{1}{b+1}\geq\frac{1}{a}+\frac{1}{b}$. \hfill ($\otimes$)
		
		Let $u$ be a $k$-vertex which is incident with faces $f_1,f_2,\cdots,f_{k}$.
		According to rule R1, $u$ totally receives charge $S=\sum_{i=1}^k\frac{d_G(f_i)-4}{d_G(f_i)}=k-4\sum_{i=1}^k\frac{1}{d_G(f_i)}$ from its incident faces.
		The supposition $r\geq \frac{15}{2}$ implies that not all of $f_1,\ldots,f_k$ are triangles. It follows by ($\otimes$) that
		$\sum_{i=1}^k\frac{1}{d_G(f_i)}$ reaches its maximum when all of $f_1,\ldots,f_k$ are triangles except one.
		Since $\sum_{i=1}^k d_G(f_i)\geq \lceil rk\rceil$, we have $S\geq k-4(\frac{1}{3}(k-1)+\frac{1}{\lceil rk\rceil-3(k-1)})= \frac{4-k}{3}-\frac{4}{\lceil rk \rceil-3k+3}$. In particular, if $u$ is incident with at most two triangles, then  we have $S\geq k-4(\frac{2}{3}+\frac{1}{4}(k-3)+\frac{1}{\lceil rk\rceil-6-4(k-3)})=\frac{1}{3}-\frac{4}{\lceil rk \rceil -4k+6}.$
	\end{proof}

	\begin{claim} \label{C2}
		The charge that a 5-vertex sends to a 4-neighbor by R2.3 is smaller than or equal to the charge that a 5-vertex sends to a 5-neighbor which is adjacent to a 2-vertex by R2.4, that is, $\frac{4}{n(\lceil 4r \rceil-9)}\leq\frac{1}{3}(\frac{4}{\lceil 5r \rceil-12}+\frac{2}{\lceil 2r \rceil-3})$.
	\end{claim}
	
	\begin{proof} Since $\frac{4}{n(\lceil 4r \rceil-9)}\leq\frac{2}{\lceil 4r \rceil-9}\leq \frac{2}{4r-9}$ and $\frac{1}{3}(\frac{4}{ 5r+1-12}+\frac{2}{ 2r+1-3})\leq\frac{1}{3}(\frac{4}{\lceil 5r \rceil-12}+\frac{2}{\lceil 2r \rceil-3})$, we only need to prove that $\frac{2}{4r-9}\leq \frac{1}{3}(\frac{4}{5r+1-12}+\frac{2}{2r+1-3})$, which is equivalent to $2r^2-15r+23\geq 0$ by simplification.
		Clearly, this inequality is true for every $r\geq 5 + \frac{2}{5}$.
	\end{proof}

	It remains to check the final charge for all $x\in V\cup F$.
	
	Let $f\in F$, then $ch^{*}(f)\geq d_G(f)-4-d_G(f) \frac{d_G(f)-4}{d_G(f)}=0$ by R1.
	
	Let $v\in V$.
	If $d_G(v)=2$, then $v$ receives at least $\frac{2}{3}-\frac{4}{\lceil 2r \rceil-3}$ in total from its incident faces by Claim \ref{C1}.
	By Lemma \ref{trivial_statement}, $v$ has two 5-neighbors.
	Thus, $v$ receives $\frac{2}{3}+\frac{2}{\lceil 2r \rceil-3}$ from each of them by R2.1. So we have
	$ch^{*}(v) \geq d_G(v)-4+(\frac{2}{3}-\frac{4}{\lceil 2r\rceil-3})+2 (\frac{2}{3}+\frac{2}{\lceil 2r \rceil-3})=0$.
	
	If $d_G(v)=3$, then $v$ receives at least $\frac{1}{3}-\frac{4}{\lceil 3r \rceil-6}$ in total from its incident faces by Claim \ref{C1}. By Lemmas \ref{trivial_statement} and \ref{VAL}, $v$ has three $4^+$-neighbors, and two of them have degree 5.
	If $v$ has a $4$-neighbor, then by R2.2.1, $ch^{*}(v)\geq d_G(v)-4+(\frac{1}{3}-\frac{4}{\lceil 3r\rceil-6})+2 (\frac{1}{3}+\frac{2}{\lceil 3r \rceil-6})=0$.
	Otherwise, by R2.2.2, $ch^{*}(v)\geq d_G(v)-4+(\frac{1}{3}-\frac{4}{\lceil 3r\rceil-6})+3 (\frac{2}{9}+\frac{4}{3(\lceil 3r \rceil-6)})=0$.
	
	If $d_G(v)=4$, then $v$ receives at least $-\frac{4}{\lceil 4r \rceil-9}$ in total from its incident faces by Claim \ref{C1}. Say $v$ has precisely $n$ 5-neighbors.  By Lemma \ref{trivial_statement}, we have $2\leq n\leq 4$. By R2.3, each of these 5-neighbors send $\frac{4}{n(\lceil 4r \rceil-9)}$ to $v$.
	Therefore, $ch^{*}(v)\geq d_G(v)-4-\frac{4}{\lceil 4r \rceil-9}+n \frac{4}{n(\lceil 4r \rceil-9)}=0$.
	
	If $d_G(v)=5$, then $v$ receives at least $-\frac{1}{3}-\frac{4}{\lceil 5r \rceil-12}$ in total from its incident faces by Claim \ref{C1}.
	First assume $v$ has a 2-neighbor, then by Lemma \ref{Zhang_lemma}, $v$ has four 5-neighbors and at least three of them are adjacent to no $3^-$-vertex.
	Hence, by R2.1 and R2.4, $ch^{*}(v)\geq d_G(v)-4-(\frac{1}{3}+\frac{4}{\lceil 5r\rceil-12})-(\frac{2}{3}+\frac{2}{\lceil 2r\rceil-3})+3 (\frac{1}{3}(\frac{4}{\lceil 5r\rceil-12}+\frac{2}{\lceil 2r \rceil-3}))=0$.
	
	Next assume that $v$ has a 3-neighbor $u$, then by Lemma \ref{VAL}, $v$ has at least three 5-neighbors. In this case, $v$ sends nothing to each 5-neighbor.
	Let $w$ be the remaining neighbor of $v$. Then $d_G(w)\in \{3,4,5\}$.
	
	If $d_G(w)=3$, then $uw\notin E(G)$ by Lemma \ref{trivial_statement}. Furthermore,
	Lemma \ref{SZ} implies that neither $vw$ nor $uv$ is contained in a triangle. It follows that $v$ is incident with at most two triangles. Thus, by Claim \ref{C1}, $v$ receives a charge of
	at least $\frac{1}{3}-\frac{4}{\lceil 5r \rceil -14}$ in total from its incident faces. Moreover, both $u$ and $w$ have no $4^-$-neighbors. Suppose to the contrary that $t$ is a $4^-$-neighbor of $u$ (analogously of $w$). By Lemma \ref{trivial_statement}, we have $d_G(t)=4$. By applying Lemma \ref{Zhang_lemma} to $ut$,
	we have $d_G(w)\geq 4$, a contradiction.
	Hence, $v$ sends $\frac{2}{9}+\frac{4}{3(\lceil 3r \rceil-6)}$ to each of $u$ and $w$ by rule R2.2.2, yielding
	$ch^{*}(v)\geq d_G(v)-4+(\frac{1}{3}-\frac{4}{\lceil 5r \rceil -14})-2(\frac{2}{9}+\frac{4}{3(\lceil 3r \rceil-6)})=\frac{8}{9}-\frac{4}{\lceil 5r\rceil-14}-\frac{8}{3(\lceil 3r \rceil -6)}$.
	
	If $d_G(w)=4$, and if $u$ is adjacent to $w$, then by Lemma \ref{Zhang_lemma}, $w$ has three 5-neighbors. Hence,
	by R2.2 and R2.3, $ch^{*}(v)\geq d_G(v)-4-(\frac{1}{3}+\frac{4}{\lceil 5r\rceil-12})-(\frac{1}{3}+\frac{2}{\lceil 3r\rceil-6})-\frac{4}{3(\lceil 4r \rceil-9)} =
	\frac{1}{3}-\frac{2}{\lceil 3r\rceil-6}-\frac{4}{3(\lceil 4r \rceil-9)}-\frac{4}{\lceil 5r\rceil-12}$.
	If $u$ is not adjacent to $w$, then for any neighbor $t$ of $u$, we have $d_G(t)\geq 4$ by Lemma \ref{trivial_statement}.
	If $d_G(t)=4$, then by applying Lemma \ref{Zhang_lemma} to $ut$ we have $d_G(w)=5$, a contradiction.
	Hence, $d_G(t)=5$. This means all neighbors of $u$ are of degree 5.
	By R2.2.2,
	$ch^{*}(v)\geq d_G(v)-4-(\frac{1}{3}+\frac{4}{\lceil 5r\rceil-12})-(\frac{2}{9}+\frac{4}{3(\lceil 3r\rceil-6)})-\frac{2}{\lceil 4r \rceil-9} =
	\frac{4}{9}-\frac{4}{3(\lceil 3r\rceil-6)}-\frac{2}{\lceil 4r \rceil-9}-\frac{4}{\lceil 5r\rceil-12}$.
	
	If $d_G(w)=5$, then $v$ sends charge only to $u$. Hence,
	$ch^{*}(v)\geq d_G(v)-4-(\frac{1}{3}+\frac{4}{\lceil 5r\rceil-12})-(\frac{1}{3}+\frac{2}{\lceil 3r\rceil-6}) = \frac{1}{3}-\frac{2}{\lceil 3r\rceil-6}-\frac{4}{\lceil 5r\rceil-12}$.
	
	It remains to consider the case when $v$ has five $4^+$-neighbors.
	In this case it follows with Claim \ref{C2} that $ch^{*}(v)\geq d_G(v)-4-(\frac{1}{3}+\frac{4}{\lceil 5r \rceil -12})-5 (\frac{1}{3}(\frac{4}{\lceil 5r\rceil-12}+\frac{2}{\lceil 2r\rceil-3})) =
	\frac{2}{3}-\frac{32}{3(\lceil 5r\rceil-12)}-\frac{10}{3(\lceil 2r\rceil-3)}$.
	
	Since $r>7+\frac{1}{2}$ it follows that $ch^{*}(x)\geq 0$ for all $x\in V\cup F$.
\end{proof}

The result for $k=6$ in Theorem \ref{MAIN} is implied by the following theorem.

\begin{theorem} \label{thm_delta6_F*}
	If $G$ is a planar $6$-critical graph, then $F^*(G) \leq 3 + \frac{2}{5}$.
\end{theorem}

\begin{proof}
	Suppose to the contrary that $F^*(G) > 3 + \frac{2}{5}$.
	Let $\Sigma$ be an embedding of $G$ into the Euclidean plane and $F^*(G) = F((G,\Sigma))$.  We have
	\begin{align*}
		\sum_{f\in F(G)}(2d_G(f)-6)
		&=4|E(G)|-6|F(G)|\\
		&=4|E(G)|-6(|E(G)|+2-|V(G)|)~~(\mbox{by Euler's formula)}\\
		&=6|V(G)|-2|E(G)|-12\\
		&\leq |V(G)|-15~~~~~(\mbox{by Theorem}~\ref{Luo_etal})
	\end{align*}
	and therefore, $-|V(G)| + \sum_{f\in F(G)}(2d_G(f)-6)\leq -15$. \hfill $(\ast)$
	
	Define the initial charge $ch(x)$ for each $x\in V(G)\cup F(G)$ as follows: $ch(v)=-1$ for every $v\in V(G)$ and $ch(f)=2d_G(f)-6$ for every $f\in F(G)$.
	It follows from inequality $(\ast)$ that $\sum_{x\in V(G)\cup F(G)}ch(x)\leq-15$.
	
	A vertex $v$ is called \textit{heavy} if $d_G(v)\in \{5,6\}$ and $v$ is incident with a face of length 4 or 5.
	A vertex $v$ is called \textit{light} if $2\leq d_G(v)\leq 4$ and $v$ is incident with no $6^+$-face and with at most one $4^+$-face.
	We say a light vertex $v$ is \textit{bad-light} if $v$ has a neighbor $u$ such that $d_G(u)+d_G(v)=8$, and \textit{good-light} otherwise.
	
	Discharge the elements of $V(G)\cup F(G)$ according to following rules.\\
	{\bf R1:}~~every $4^+$-face $f$ sends $\frac{2d_G(f)-6}{d_G(f)}$ to each incident vertex.\\
	{\bf R2:}~~every heavy vertex sends $\frac{3}{10}$ to each bad-light neighbor, and $\frac{1}{10}$ to each good-light neighbor.
	
	Let $ch^*(x)$ denote the final charge of each $x\in V(G)\cup F(G)$ after discharging.
	On one hand, the sum of charge over all elements of $V(G)\cup F(G)$ is unchanged. Hence, we have
	$\sum_{x\in V(G)\cup F(G)}ch^*(x)\leq-15.$ On the other hand, we show that $ch^*(x) \geq 0$ for every $x\in V(G)\cup F(G)$ and hence, this obvious contradiction completes the proof.
	
	It remains to show that $ch^*(x) \geq 0$ for every $x \in V(G) \cup F(G)$.
	
	Let $f\in F(G)$.
	If $d_G(f)=3$, then no rule is applied for $f$. Thus, $ch^*(f)=ch(f)=0.$
	
	If $d_G(f)\geq 4$, then by R1 we have $ch^*(f)=ch(f)-d_G(f)  \frac{2d_G(f)-6}{d_G(f)}=0.$
	
	Let $v\in V(G)$. First we consider the case when $v$ is heavy. On one hand,
	since $F((G,\Sigma)) > 3 + \frac{2}{5}$, it follows that either $v$ is incident with a $5^+$-face and another $4^+$-face or $v$ is incident with at least three 4-faces.
	In both cases, $v$ receives at least $\frac{13}{10}$ in total from its incident faces by R1. On the other hand, we claim that $v$ sends at most $\frac{3}{10}$ out in total.
	If $v$ is adjacent to a bad-light vertex $u$, then all other neighbors of $v$ have degree at least 5 by Lemma \ref{Zhang_lemma}. Hence, $v$ sends $\frac{3}{10}$ to $u$ by R2 and nothing else to its other neighbors.
	If $v$ is adjacent to no bad-light vertex, then $v$ has at most three good-light neighbors by Lemma \ref{VAL}. Hence, $v$ sends $\frac{1}{10}$ to each good-light neighbor by R2 and nothing else
	to its other neighbors. Therefore, $ch^*(v)\geq ch(v)+\frac{13}{10}-\frac{3}{10}=0.$
	
	Second we consider the case when $v$ is not heavy. In this case, $v$ sends no charge out.
	If $v$ is incident with a $6^+$-face, then $v$ receives at least $1$ from this $6^+$-face by R1. This gives $ch^*(v)=ch(v)+1=0.$
	If $v$ is incident with at least two $4^+$-faces, then $v$ receives at least $\frac{1}{2}$ from each of them by R1. This gives $ch^*(v)=ch(v)+ \frac{1}{2} + \frac{1}{2}=0.$
	We are done in both cases above. Hence, we may assume that $v$ is incident with no $6^+$-face and with at most one $4^+$-face. From $F((G,\Sigma)) > 3 + \frac{2}{5}$ it follows that $v$ is incident to a face $f_v$ such that $d_G(f_v)\in \{4,5\}$.
	Since $v$ is not heavy, $2\leq d(v)\leq 4$.
	Hence, $v$ is light by definition. We distinguish two cases by the length of $f_v$.
	
	If $d_G(f_v)=4$, then by the fact that $F^*(G)\geq 3+\frac{2}{5}$, we have $d_G(v)=2$. By Lemma \ref{trivial_statement}, both neighbors of $v$ are heavy and $v$ is bad-light. Thus, $v$ receives $\frac{1}{2}$ from $f_v$ by R1 and $\frac{3}{10}$ from each neighbor by R2, yielding $ch^*(v)=ch(v)+\frac{1}{2}+ \frac{3}{10} + \frac{3}{10}>0$.
	
	If $d_G(f_v)=5$, then $v$ receives $\frac{4}{5}$ from $f_v$.
	If $v$ is not a bad-light 4-vertex, then Lemma \ref{trivial_statement} implies that each neighbor of $v$ has degree 5 or 6. Hence,
	both of the two neighbors of $v$ contained in $f_v$ are heavy. By R2, each of them sends charge at least $\frac{1}{10}$ to $v$, and therefore, $ch^*(v)\geq ch(v)+\frac{4}{5}+ \frac{1}{10} +  \frac{1}{10}=0$.
	If $v$ is a bad-light 4-vertex, then
	Lemma \ref{VAL} implies that at least one of the two neighbors of $v$ contained in $f_v$ is heavy. Thus, this heavy neighbor sends charge $\frac{3}{10}$ to $v$,
	and therefore, $ch^*(v)\geq ch(v)+\frac{4}{5}+\frac{3}{10}>0$.
\end{proof}

\section{Concluding remarks}

Recently, Cranston and Rabern \cite{Cranston_2015} improved Jakobsen's result  (Theorem \ref{Jakobsen}) on the
lower bound on the number of edges in a 3-critical graph. They gave a computer-aided proof of the following statement.

\begin{theorem}[\cite{Cranston_2015}]
	Every $3$-critical graph $G$, other than the Petersen graph with a vertex deleted, has $|E(G)|\geq \frac{50}{37}|V(G)|.$
\end{theorem}

Hence, $|E(G)|\geq \frac{50}{37}|V(G)|$ for every planar 3-critical graph. By a similar argument as in the proof of
Proposition \ref{pro_globe}, this result improves the bound of $\overline{b}_3$ from $6\leq \overline{b}_3< 8$  to $6\leq \overline{b}_3< \frac{100}{13}.$

However, the precise values of these parameters are unclear.

\begin{problem}
	What are the precise values of $\overline{b}_k$ and $b_k^*$?
\end{problem}

By Proposition \ref{pro_globe}, $\overline{F}(G)$ has an upper bound for every critical planar graph $G$.
However, this is not always true for class 2 planar graphs. Similarly, Theorems \ref{thm_delta5_F*} and \ref{thm_delta6_F*} can not be generalized
to class 2 planar graphs.

\end{document}